\documentclass{amsart}
\usepackage{amssymb,
enumitem,
mathrsfs,
mathtools,
hyperref,
tikz,
upgreek,
verbatim
}
\hypersetup{colorlinks=true}
\numberwithin{equation}{section}
\usepackage[T1]{fontenc}
\usepackage[shadow]{todonotes}

\renewcommand{\L}{\mathcal{L}}
\newcommand{\RR}{\mathbb{R}}
\newcommand{\NN}{\mathbb{N}}

\newcommand{\pre}[2]{{}^{#1} #2}
\newcommand{\set}[2]{\{ #1 \mid #2 \}}

\newcommand{\dom}{\operatorname{dom}}

\newcommand{\embeds}{\sqsubseteq}
\newcommand{\analytic}{\boldsymbol{\Sigma}_1^1}
\newcommand{\Aut}[1]{\mathrm{Aut}(#1)}
\newcommand{\Stab}[1]{\mathrm{Stab}(#1)}
\newcommand{\Subg}[1]{\mathrm{Subg}(#1)}
\newcommand{\SG}[1]{\mathrm{SG}(#1)}
\newcommand{\Homeo}[1]{\mathrm{Homeo}(#1)}
\newcommand{\Isom}[1]{\mathrm{Isom}(#1)}

\newenvironment{enumerate-(a)}{\begin{enumerate}[label={\upshape (\alph*)}, leftmargin=2pc]}{\end{enumerate}}
\newenvironment{enumerate-(a)-r}{\begin{enumerate}[label={\upshape (\alph*)}, leftmargin=2pc,resume]}{\end{enumerate}}
\newenvironment{enumerate-(a)-5}{\begin{enumerate}[label={\upshape (\alph*)}, leftmargin=2pc,start=5]}{\end{enumerate}}
\newenvironment{enumerate-(A)}{\begin{enumerate}[label={\upshape (\Alph*)}, leftmargin=2pc]}{\end{enumerate}}
\newenvironment{enumerate-(A)-r}{\begin{enumerate}[label={\upshape (\Alph*)}, leftmargin=2pc,resume]}{\end{enumerate}}
\newenvironment{enumerate-(i)}{\begin{enumerate}[label={\upshape (\roman*)}, leftmargin=2pc]}{\end{enumerate}}
\newenvironment{enumerate-(i)-r}{\begin{enumerate}[label={\upshape (\roman*)}, leftmargin=2pc,resume]}{\end{enumerate}}
\newenvironment{enumerate-(I)}{\begin{enumerate}[label={\upshape (\Roman*)}, leftmargin=2pc]}{\end{enumerate}}
\newenvironment{enumerate-(I)-r}{\begin{enumerate}[label={\upshape (\Roman*)}, leftmargin=2pc,resume]}{\end{enumerate}}
\newenvironment{enumerate-(1)}{\begin{enumerate}[label={\upshape (\arabic*)}, leftmargin=2pc]}{\end{enumerate}}
\newenvironment{enumerate-(1)-r}{\begin{enumerate}[label={\upshape (\arabic*)}, leftmargin=2pc,resume]}{\end{enumerate}}
\newenvironment{itemizenew}{\begin{itemize}[leftmargin=2pc]}{\end{itemize}}

\newtheorem{theorem}{Theorem}[section]
\newtheorem{lemma}[theorem]{Lemma}
\newtheorem{corollary}[theorem]{Corollary}
\newtheorem{proposition}[theorem]{Proposition}

\newtheorem{claim}{Claim}[theorem]
\theoremstyle{definition}
\newtheorem{definition}[theorem]{Definition}

\theoremstyle{remark}
\newtheorem{remark}[theorem]{Remark}

\begin{document}

\title[Universality of group embeddability]{Universality of group embeddability}
\date{\today}
\author[F.~Calderoni]{Filippo Calderoni}
\author[L.~Motto Ros]{Luca Motto Ros}

\address{Dipartimento di matematica \guillemotleft{Giuseppe Peano}\guillemotright, Universit\`a di Torino, Via Carlo Alberto 10, 10121 Torino --- Italy}
\email{filippo.calderoni@unito.it}

\address{Dipartimento di matematica \guillemotleft{Giuseppe Peano}\guillemotright, Universit\`a di Torino, Via Carlo Alberto 10, 10121 Torino --- Italy}
\email{luca.mottoros@unito.it}

 \subjclass[2010]{Primary: 03E15}
 \keywords{Borel reducibility; countable groups; Polish groups; separable metric groups; group embeddability}
\thanks{We thank Rapha\" el Carroy for many useful comments. The first author also thanks Gabriel Debs, Dominique Lecomte, 
Fran\c cois Lemaitre and Alain Louveau for their comments and their keen interest in this work at the Descriptive Set Theory working group in Paris. 
The second author was supported for this research by the Young
Researchers Program ``Rita Levi Montalcini'' 2012 through the project ``New
advances in Descriptive Set Theory''. The paper was completed during the ESI workshop ``Current trends in Descriptive Set Theory'' 2016 in Vienna: the authors would like to deeply thank the Erwin Schr\"oder International Institute for Mathematics and Physics (ESI) for their support in that occasion.}

\begin{abstract} 
Working in the framework of Borel reducibility, we study various notions of embeddability between groups. We prove that the embeddability between countable groups, the topological embeddability between (discrete) Polish groups, and the isometric embeddability between separable groups with a bounded bi-invariant complete metric are all invariantly universal analytic quasi-orders. This strengthens some results from \cite{Wil14} and~\cite{FerLouRos}.
\end{abstract}

\maketitle

\section{Introduction}

We work in the framework of analytic (i.e.\ \( \analytic \)) equivalence relations, that is we consider
pairs $(X,E)$ consisting of a standard Borel space $X$ together with an equivalence relation  $E$ on it which is analytic as a subset of $X^2$ (we refer the reader to the beginning of Section~\ref{sec:preliminaries} for the definitions of standard Borel spaces, analytic sets, and other preliminary notions). Since the seminal papers~\cite{FriSta,HarKecLou}, analytic equivalence relations are usually compared via the quasi-order of \emph{Borel reducibility}: if $(X,E)$ and $(Y,F)$ are \( \analytic \) equivalence relations, we say that $E$ is \emph{Borel reducible} to $F$  ($E\leq_B F$ in symbols) if there is
a Borel function $f \colon X\to Y$ such that  for every $x,x'\in X$
\[
{x\mathrel{E}x'} \iff {{f(x)}\mathrel{F}{f(x')}}.
\]
We also write \( E \sim_B F \) when the two \( \analytic \) equivalence relations \( E \) and \( F \) are \emph{Borel bi-reducible}, i.e.\ \( E \leq_B F \) and \( F \leq_B E \).

Over the last two decades Borel reducibility played a prominent role in the field of descriptive set theory: on the one hand it provides an efficient tool for measuring the complexity of various classification problems arising from different areas of mathematics (see for instance~\cite{Gao2003,FerLouRos,Sab}); on the other hand, the abstract analysis of the structure of \( \analytic \) equivalence relations under \( \leq_B \) turned out to be extremely challenging, yielding to a great variety of results and sophisticated techniques involving e.g.\ measure theory, ergodic theory, and so on. 

Harrington was the first to point out the existence of $\leq_B$-maxima%
\footnote{Clearly, by definition of maximum all these elements are Borel bi-reducible to each other.}
in the class of all analytic equivalence relations: such elements are called \emph{complete}, 
and can be regarded as the most complicated analytic equivalence relations --- any assignment of complete invariants for them can 
be turned into an assignment of complete invariants for any other \( \analytic \) equivalence relation.
Harrington's example came from an abstract construction specifically designed to obtain a complete \( \analytic \) equivalence relation, but some years later Louveau and Rosendal isolated in~\cite{LouRos} many 
natural examples 
coming from various areas of mathematics. The approach undertaken in~\cite{LouRos} consists in studying analytic \emph{quasi-orders}, i.e. reflexive 
and transitive binary relations (rather than equivalence relations). One can extend the notion 
of Borel reducibility to this broader context \emph{verbatim}: given two \( \analytic \) quasi-orders \((X,P)\) and \((Y,Q)\), 
we say that $P$ is \emph{Borel reducible} to $Q$  ($P\leq_B Q$ in symbols) if there is
a Borel function $f \colon X\to Y$ such that  for every $x,x'\in X$
\[
{x\mathrel{P}x'} \iff {{f(x)}\mathrel{Q}{f(x')}}.
\]
In the mentioned paper it is shown that, up to a natural coding, the embeddability relation 
between countable graphs is \( \leq_B \)-above (i.e.\ \emph{complete} for) all analytic quasi-orders, and this easily implies that the bi-embeddability relation between countable graphs is a complete \( \analytic \) 
equivalence relation. (Indeed, it can be shown that every complete \( \analytic \) equivalence relation can be construed as the symmetrization of a \( \analytic \) quasi-order which is complete in its class, so the technique of 
Louveau and Rosendal is as general as possible.)

A strengthening of completeness for \( \analytic \) quasi-orders was isolated in~\cite{FriMot}, where it is shown that for any \( \analytic \) quasi-order \( R \) there is an \( \L_{\omega_1 \omega} \)-sentence \( \upvarphi \) (all of whose models are graphs) such that \( R \) is Borel bi-reducible to embeddability between countable models of \( \upvarphi \). This means that not only the embeddability relation on countable graphs is as complicated as possible, but also that it has a stronger universality property: it contains in a natural way (i.e.\ as an
 \( \L_{\omega_1 \omega} \)-elementary
 subclass\footnote{An \( \L_{\omega_1 \omega} \)-elementary class (of countable structures) is the collection of countable models of a given sentence in the infinitary logic \( \L_{\omega_1 \omega} \).}) a faithful copy of every \( \analytic \) quasi-order. 
The above result naturally leads to the following definition of invariant universality. (By the Lopez-Escobar theorem~\cite[Theorem 16.8]{Kec}, if \( X \) is a space of countable structures and \( E \) is the isomorphism relation on it, than any \( B \) as in Definition~\ref{Definition : invariantly universal} is an \( \L_{\omega_1 \omega} \)-elementary class.)

 \begin{definition}[{\cite{CamMarMot}}]\label{Definition : invariantly universal}
Let $S$ be a $\analytic$ quasi-order on some standard Borel space $X$ and let $E$ be a $\analytic$ equivalence subrelation of $S$. We say that $(S,E)$
is \emph{invariantly universal} (or $S$ is invariantly universal with respect to $E$) if for every $\analytic$ quasi-order $R$ there is a Borel subset $B\subseteq X$ which is invariant with respect to $E$ and
such that the restriction of $S$ to $B$ is Borel bi-reducible with $R$.
 \end{definition}
 
 It immediately follows from the definition that if \( (S,E) \) is invariantly universal, then \( S \) is a complete \( \analytic \) quasi-order.
 Of course, to avoid trivial pathologies one should consider only ``meaningful'' pairs: for example, when \( S \) is induced by some notion of morphism between certain objects, it makes sense to pair it with the equivalence relation induced by the associated notion of isomorphism. In particular, when \( S \) is some kind of embeddability relation, then \( E \) is usually taken to be the induced isomorphism relation: when this is the case, we drop the reference to \( E \) and simply say that \( S \) (instead of the pair \( (S,E) \)) is invariantly universal.

The notion of invariant universality has been extensively studied in the papers~\cite{CamMarMot,CamMarMot16}. Perhaps against intuition, 
it turned out to be a quite widespread phenomenon: essentially, all the \( \analytic \) quasi-orders that were known to 
be complete turned out to be invariantly universal when paired with the naturally associated equivalence relation (including, for example, embeddability 
between graphs, topological embeddability between compacta, isometric embeddability between metric spaces, and linear isometric embeddability between Banach spaces). 
Despite the great diversity of the examples considered, all these invariant universality results were obtained via a unique technique, which can be applied only when the equivalence relation \( E \) is Borel reducible to an orbit equivalence relation (see Theorem~\ref{Theorem : CMMR13} below). In all the above mentioned situations, this extra condition was granted for free, but it can become a serious obstacle when \( E \) is e.g.\ a complete \( \analytic \) equivalence relation, as it is the case when considering topological groups. Indeed, the (topological) embeddability relation between Polish groups is complete by~\cite[Corollary 34]{FerLouRos}, but the same proof also shows that the (topological) isomorphism relation between Polish groups is a complete \( \analytic \) equivalence relation. Thus, on the one hand the completeness of embeddability invites to check whether it is indeed invariantly universal (when paired with the isomorphism relation), on the other hand the completeness of the isomorphism relation seems to forbid the use of the only known technique for proving invariant universality, a situation we are facing for the very first time.

In this paper, we will confirm the general trend uncovered in~\cite{CamMarMot,CamMarMot16} (``all complete \( \analytic \) quasi-orders are indeed invariantly universal'') by showing that also the embeddability relation 
between Polish groups is  invariantly universal. To overcome the technical difficulty explained above, we use a construction due to J.\ Williams who showed in~\cite{Wil14} (using small cancellation theory techniques) that the 
embeddability relation between countable graphs Borel reduces to the embeddability relation between countable groups, so that the latter is complete for \( \analytic \) quasi-orders. After introducing some preliminary notions and 
results in Section~\ref{sec:preliminaries}, in Section~\ref{sec:ctblegroups} we strengthen Williams' result by showing that the embeddability relation between countable groups is in fact invariantly universal 
(Theorem~\ref{Theorem : Gp uni}), a result which may be of independent interest. In Section~\ref{sec:topgroups} we 
then show how to adapt this construction to deal with Polish groups (Subsection~\ref{sec:polishgroups}) and separable groups endowed with a 
complete bi-invariant metric (Subsection~\ref{sec:metricgroups}): in all these cases, we obtain that the relevant embeddability relation is invariantly universal 
(Theorems~\ref{Theorem : PGp uni},~\ref{Theorem : embeds_i invariantly universal}, and~\ref{thm:continuouslogic}).

\section{Preliminaries} \label{sec:preliminaries}
A topological space X is \emph{Polish} if it is separable and completely metrizable.
If \(A\) is a countable set, the spaces \(2^{A}\) and \(\NN^{A}\) viewed as the product of infinitely many copies
of \(2\) and \(\NN\) with the discrete topology, respectively, are Polish.
In this paper we mainly deal with spaces of the form \(2^{\NN^{n}}\), for some integer \(n\), or \( 2^G \), where \( G \) is a countable group.
A \emph{Polish group} is a topological group whose topology is Polish. 
A well known example is \(S_{\infty}\), the group of all bijections from \(\NN\) to \(\NN\),
which is a \(G_{\delta}\) subset of \(\NN^{\NN}\) and a Polish group with the relative topology.

A \emph{standard Borel space} is a pair $(X,\mathcal B)$ such that $\mathcal B$ is the $\sigma$-algebra of Borel subsets of $X$ with respect to some Polish topology on $X$. 
Given a Polish space $X$, the space \( F(X) \) of closed subsets of \( X \) is a standard Borel space when equipped with the Effros Borel structure (see~\cite[Section 12.C]{Kec}). 
If $\mathbf{G}$ is a Polish group, then the space $\Subg{\mathbf{G}}$ of closed subgroups of $\mathbf{G}$ is a Borel subset of \( F(\mathbf{G}) \), and thus it is standard Borel as well.

A subset of a standard Borel space is  $\analytic$, or \emph{analytic},  if it is the image of a standard Borel space via a Borel function.
In particular, a binary relation defined on a standard Borel space $X$ is \emph{analytic} if it is a \( \analytic \) subset of $X\times X$. A \textit{co-analytic} set is a subset of a standard Borel space whose complement is analytic.

A \emph{quasi-order} is a reflexive and transitive binary relation. Any quasi-order $Q$ on a set $X$ canonically induces an equivalence relation on $X$, which is denoted by $E_Q$, defined by setting $x \mathrel{E_Q} y$ if and only if $x \mathrel{Q} y$ and $y \mathrel{Q} x$  (for all \( x,y \in X \)). If \( Q \) is analytic, then so is \( E_Q \).

If a Polish group $\boldsymbol{G}$ acts on a standard Borel space $X$ in a Borel way, then we say that $X$ is a \emph{standard Borel $\boldsymbol{G}$-space} and
we denote by $E_{\boldsymbol{G}}^X$ the \emph{orbit equivalence relation} induced by the action of $\boldsymbol{G}$ on $X$. Such equivalence relation is analytic. The \textit{stabilizer} of a point \( x\in X \) is the subgroup
\[
\Stab{x}\coloneqq\set{g\in \boldsymbol{G}}{g\cdot x=x},
\]
where \( g\cdot x \) denotes the value of the action on the pair $(g,x)$.

Given two binary relations $R$ and $R'$ on  standard Borel spaces $X$ and $Y$, respectively, we say that $R$ \emph{Borel reduces} (or is \emph{Borel reducible}) to $R'$ (in symbols, \( R \leq_B R' \)) if and only if there is a Borel function $f \colon X\to Y$ such that for every $x,y\in X$
\[
x \mathrel{R} y \iff f(x)\mathrel{R'}f(y).
\]
Such an \( f \) is called (\emph{Borel}) \emph{reduction} (of \( R \) to \( R' \)). The relations \( R \) and \( R' \) are \emph{Borel bi-reducible} (in symbols, \( R \sim_B R' \)) if \( R \leq_B R' \) and \( R' \leq_B R \).

Louveau and Rosendal proved in \cite{LouRos} that among all $\analytic$ quasi-orders there are $\leq_B$-maximum elements: such quasi-orders are (by definition of maximum) Borel bi-reducible to each other, and are called \emph{complete $\analytic$ quasi-orders}. In~\cite{LouRos} the authors proved that several $\analytic$ quasi-orders which naturally  occur in mathematics are indeed complete: among those, the first prominent example is the embeddability relation between countable graphs that we briefly describe below. Let $X_{Gr}$ be the space of graphs on \( \NN \). By identifying
each graph with the characteristic function of its edge relation, $X_{Gr}$ can be construed as a closed subset of $2^{\mathbb{N}^2}$, and thus it is a Polish space. Given $S,T\in X_{Gr}$, set 
$S\embeds_{Gr}T$ if and only if $S$ embeds into $T$, i.e.\ if and only if
there is an injective function $f \colon \NN \to \NN$ such that $m$ and \( n \) are adjacent in \( S \) if and only if \( f(m) \) and \( f(n) \) are adjacent in \( T \) (for every $m,n\in\NN$).

\begin{theorem}[{\cite[Theorem 3.1]{LouRos}}]\label{Theorem : LouRos}
The relation $\embeds_{Gr}$ on \( X_{Gr} \) of  embeddability between countable graphs is a complete $\analytic$ quasi-order.
\end{theorem}

In \cite{FriMot} and \cite{CamMarMot}, the authors modified the proof of Theorem~\ref{Theorem : LouRos}
in order to find a Borel $\mathbb{G}\subseteq X_{Gr}$ with the following properties:
\begin{enumerate-(i)}
\item each element of $\mathbb{G}$ is a combinatorial tree (i.e.\ a connected acyclic graph);
\item the equality and isomorphism relations restricted to $\mathbb{G}$, denoted respectively by $=_{\mathbb{G}}$ and $\cong_{\mathbb{G}}$, coincide;
\item each graph in $\mathbb{G}$ is rigid, i.e.\ it has no nontrivial automorphism;
\item $\embeds_{\mathbb{G}}$, the restriction of $\embeds_{Gr}$ to $\mathbb{G}$, is a complete $\analytic$ quasi-orders.
\end{enumerate-(i)}

The standard Borel space $\mathbb{G}$ is used to test whether a pair $(Q,E)$ satisfying the conditions of Definition \ref{Definition : invariantly universal} is invariantly universal. In fact, the following result gives sufficient conditions to ensure the invariant universality of a pair.

\begin{theorem}[{\cite[Theorem 4.2]{CamMarMot}}]\label{Theorem : CMMR13} 
Suppose that $Q$ is a $\analytic$ quasi-order on a standard Borel space $X$ and let
$E\subseteq Q$ be a $\analytic$ equivalence relation on \( X \).
Then, $(Q,E)$ is invariantly universal
provided that the following conditions hold:
\begin{enumerate-(i)}\label{enumerate : TheoremCMMR13}
 \item \label{condition : 1} there is a Borel reduction $f \colon \mathbb{G}\rightarrow X$ of $\embeds_\mathbb{G}$ to $Q$;
\item \label{condition : 2} $f$ is also a Borel reduction of ${=_\mathbb{G}}$ (equivalently, of \( \cong_{\mathbb{G}} \)) to ${E}$;
\item \label{condition : 3} there are a co-analytic \( E \)-invariant \( \mathrm{rng}(f) \subseteq Z \subseteq X \), a Polish group \( \boldsymbol H\), a standard Borel $\boldsymbol{H}$-space $Y$,  and a Borel reduction $h \colon Z \to Y$
of $E \restriction Z$ to $E_{\boldsymbol{H}}^Y$ such that the map
\[
 \mathbb{G}\to\Subg {\boldsymbol H}, \qquad 	T\mapsto \Stab{h(f(T))}
\] 
 is Borel.
\end{enumerate-(i)}
\end{theorem}

Recall that \( Z\subseteq X\) is \emph{\(E\)-invariant} if it is a union of \( E \)-classes. 
Notice that in the original formulation of Theorem \ref{Theorem : CMMR13}  (cf.~\cite[Theorem 4.2]{CamMarMot}) the set \(Z\) is required to be Borel, which seems a stronger condition. However, our statement is equivalent to the original one because if \(  Z \) is co-analytic
and \( E \)-invariant, then by the separation theorem for analytic  \( E\)-invariant sets (see \cite[Lemma 5.4.6]{Gao}) there is an \( E \)-invariant Borel \( \mathrm{rng}(f)\subseteq Z' \subseteq Z\) which satisfies condition~\ref{condition : 3}.

If \( Y \) is an \( S_\infty \)-space of countable structures (with \( S_\infty \) acting on \( Y \) with the usual continuous logic action, so that the induced equivalence relation is the isomorphism on \( Y \)), then the stabilizer \( \Stab{y} \) of any \( y \in Y \) is the group \( \Aut{y} \) of automorphisms of \( y \). In many applications of Theorem~\ref{Theorem : CMMR13}, the situation is considerably simplified by the fact that \( X \) itself is a space of countable structures and \( E \) is the isomorphism relation: in this case, one could verify condition \ref{condition : 3} of Theorem~\ref{enumerate : TheoremCMMR13} setting $X=Z=Y$ and $h$ equal to the identity map, so that it suffices to check the Borelness of the map
\[
 \mathbb{G} \to \Subg {S_\infty}, \qquad T \mapsto \Aut{f(T)}.
\]

\section{Embeddability between countable groups} \label{sec:ctblegroups}

Let $X_{Gp}$ be the set of groups whose underlying set is $\mathbb{N}$. 
Every such group can be identified with the (characteristic function of the) graph of its operation, hence
$X_{Gp}$ can be viewed as a $G_\delta$ subset of $2^{\mathbb{N}^3}$, and thus it is a Polish space. Let $\embeds_{Gp}$ be the $\analytic$ quasi-order of embeddability on $X_{Gp}$.  Jay Williams showed in~\cite[Theorem~5.1]{Wil14} that \( {\embeds_{Gr}} \leq_B { \embeds_{Gp}} \), which combined with Theorem~\ref{Theorem : LouRos} yields the next result.

\begin{theorem}[\cite{Wil14}]\label{Theorem : Wil14}
The relation $\embeds_{Gp}$ is a complete $\analytic$ quasi-order.
\end{theorem}

The Borel reduction used in~\cite{Wil14} maps each graph $T\in X_{Gr}$
to the group \( G_T = \langle v_0, v_1, \dotsc \mid R_T \rangle \) generated by the vertices of $T$, which we denote by $\set{v_i}{i\in \mathbb{N}}$ to avoid confusion, and the following set of relators $R_T$ encoding 
the edges of $T$: for every $T\in X_{Gr}$, $R_T$ is the smallest subset of the free group on $\set{v_i}{i\in \mathbb{N}}$
which is symmetrized (i.e.\ closed under inverses and cyclic permutations, and such that all its elements are cyclically reduced) and contains the following words (for distinct \(i,j \in \NN \)):
\begin{itemizenew}
 \item $v_i^7$
\item $(v_iv_j)^{11}$, if $(v_i,v_j)\in T$
\item $(v_iv_j)^{13}$, if $(v_i,v_j)\notin T$.
\end{itemizenew}
A \emph{piece} for the group presented by \( \langle V \mid R \rangle\) is a maximal common initial segment of two distinct \( r_1, r_2 \in R\).
It is immediate to check that for every $T\in X_{Gr}$, the set \( R_T \) satisfies the following small cancellation condition:
\begin{equation} \tag{$C'\left(\frac{1}{6}\right)$}\label{eq : sixth}
\text{if $u$ is a piece and $u$ is a subword of some $r \in R$, then $|u| < \frac{1}{6}|r|$}.
\end{equation}
Groups \( \langle V \mid R \rangle \) whose set of relators \( R \) is symmetrized and satisfies the $C'\left(\frac{1}{6}\right)$ condition are called \emph{sixth groups}.

\begin{theorem}[{\cite[Theorem V.10.1]{LynSch}}]\label{Theorem : Greendlinger}
Let $G=\langle V \mid R\rangle$ be a sixth group. If $w$ represents an element of finite order in $G$, then there is some $r\in R$ of the form $r=v^n$ such that $w$ is conjugate to a power of
$v$. Thus, if $w$ is cyclically reduced, then $w$ is a cyclic permutation of some power of $v$ with $v^n\in R$ for some \( n \in \NN \). 
\end{theorem}

The next lemma  (which is already implicit in the proof of Theorem~\ref{Theorem : Wil14}) is a nice consequence of Theorem~\ref{Theorem : Greendlinger} and shows that all automorphisms of the group \( G_T \) constructed by Williams are canonically induced (up to inverses and conjugacy) by the automorphisms of the graph \( T \).

\begin{lemma}\label{Lemma : auto}
 Let $T\in X_{Gr}$ and \( \theta: G_T\to G_T\). Then $\theta\in\Aut{G_T}$ if and only if
 the following two conditions hold:
 \begin{enumerate-(i)}
 \item \label{c1} 
 $\theta(ww')=\theta(w)\theta(w')$  for all \( w,w'\in G_T \);
\item \label{c2}
there are $\rho\in \Aut{T}$, $t\in G_T$, and $\epsilon\in \{-1,1\}$ such that 
\begin{equation}\label{eq:c2}
\theta(v_i)=t v_{\rho(i)}^\epsilon t^{-1} \qquad \text{for all } i \in \NN. 
\end{equation}

\end{enumerate-(i)}
\end{lemma}

Clearly, the \( \rho \), \( t \) and \( \epsilon \) in condition~\ref{c2} are unique. 
Automorphisms \( \theta \in \Aut{G_T} \) for which $\epsilon=1$ in~\ref{c2} are called \emph{positive}, while automorphisms \( \theta \in \Aut{G_T} \) for which $\epsilon=-1$ in~\ref{c2} are called \emph{negative}.

\begin{proof} Assume first that $\theta\in \Aut{G_T}$. Condition~\ref{c1} is satisfied by definition of automorphism, so it is enough to show that condition~\ref{c2} is satisfied as well.

\begin{claim}\label{claim:fundamental}
Let \( \theta \in \Aut{G_T} \) and \( \bar{\imath} \in \NN \). Suppose that there are \( u \in G_T \) and \( \bar{k} \in \mathbb{Z} \) with \( |\bar{k}| < 7 \) such that \( \theta(v_{\bar{\imath}}) = u v^{\bar k}_{\bar{\jmath}} u^{-1} \) for some \( \bar{\jmath} \in \NN \). Then \( \bar{k} \in \{ -1,1 \} \) and there are a map \( \rho \colon \NN \to \NN \) and \(m \in \mathbb{Z} \) with \( |m| < 7 \)
such that  \( \rho(\bar \imath) = \bar \jmath \) and for all \( i \in \NN \)
\begin{equation} \label{eq:claim}
\theta(v_i) = u v_{\rho(\bar \imath)}^{m} v_{\rho(i)}^{\bar{k}} v_{\rho(\bar \imath)}^{-m}u^{-1}.
 \end{equation}
\end{claim}

\begin{proof}[Proof of the claim]
Set \( \rho(\bar{\imath}) = \bar{\jmath} \), so that~\eqref{eq:claim} is already automatically satisfied for \( i = \bar \imath \) (independently of the value of \( m \) that we will choose). Let $\theta_u\in \Aut{G_T}$ be the inner automorphism $g\mapsto u^{-1}gu$. Clearly, ${\theta_u \circ \theta} \in \Aut {G_T}$
and  $(\theta_u \circ \theta)(v_{\bar{\imath}})=v_{\rho(\bar{\imath})}^{\bar{k}}$.
For every \( i\in\mathbb{N}\setminus \{ \bar{\imath} \} \)
the element $(\theta_u\circ \theta)(v_i)$ must have order $7$ in $G_T$, hence
there are some \(\rho(i)\in \mathbb{N}\), a reduced $w\in G_T$, and $l\in\mathbb{Z}$
with ${|l|}<7$ such that
$(\theta_u\circ \theta)(v_i)=wv_{\rho(i)}^lw^{-1}$.
Possibly $w$ may start with some power of $v_{\rho(\bar{\imath})}$:
if this is the case, let $\psi_i$ be some inner automorphism
such that $(\psi_i\circ\theta_u\circ \theta)(v_i)$ does not start with $v_{\rho(\bar{\imath})}$, i.e.\ set for \( g \in G_T \)
\begin{equation} \label{eq:psi_i}
\psi_i(g) = v_{\rho(\bar{\imath})}^{-m}gv_{\rho(\bar{\imath})}^{m}
\end{equation}
for $m \in \NN$ maximal such that $w=v_{\rho(\bar{\imath})}^mw'$.
Thus \( (\psi_i\circ\theta_u\circ \theta)(v_i)= z v_{\rho(i)}^lz^{-1}\), for some reduced word \( z\) which
does not start with a power of \( v_{\rho(\bar{\imath})}\).
Now notice that
$(\psi_i\circ\theta_u\circ \theta)(v_{\bar{\imath}}v_i)=v_{\rho(\bar{\imath})}^{\bar{k}} z v_{\rho(i)}^lz^{-1}$ must have finite order (either \( 11 \) or \( 13 \), depending on whether \( (v_{\bar{\imath}}, v_i) \in T \) or not), and
it is cyclically reduced because $z^{-1}$ does not end with any power of \( v_{\rho(\bar{\imath})}\).
Consequently, by Theorem~\ref{Theorem : Greendlinger} the element $v_{\rho(\bar{\imath})}^{\bar{k}} z v_{\rho(i)}^l z^{-1}$ must be a
cyclic permutation of some power of $v_nv_m$ for some $n,m\in \mathbb{N}$, which yields in turn that $z$ must be the identity of \( G_T \). 
Therefore $(\psi_i\circ\theta_u\circ \theta)(v_{\bar{\imath}}v_i)=v_{\rho(\bar{\imath})}^{\bar{k}} v_{\rho(i)}^l$, and the order of this element is either $11$ or $13$: this implies that
\( {\rho(\bar{\imath})} \neq {\rho(i)} \) because otherwise $v_{\rho(\bar{\imath})}^{\bar{k}}v_{\rho(i)}^l$ would have order $7$.
Moreover, the only possible values for $\bar{k}$ and $l$ are
$\bar{k}=l \in \{ -1,1 \}$ because otherwise $v_{\rho(\bar{\imath})}^{\bar{k}}v_{\rho(i)}^l$ would have infinite order.

Summing up, we proved that \( \bar{k} \in \{ -1,1 \} \) and that there are a function \( \rho \colon \mathbb{N}\to \mathbb{N} \), and an inner automorphism $\psi_i$,
for every $i \in  \NN \setminus \{ \bar{\imath} \}$, such that \( \rho(i) \neq \rho(\bar \imath ) \) and 
\begin{equation} \label{eq:auto}
(\psi_i\circ\theta_u\circ \theta)(v_i)=v_{\rho(i)}^{\bar{k}}.
\end{equation} 
We now claim that $\psi_i\circ\theta_u\circ \theta=\psi_{j}\circ\theta_u\circ \theta$ for all \( i,j \in  \NN \setminus \{ \bar{\imath} \} \). To prove this, it suffices to show that 
\( \psi_i\circ\theta_u\circ \theta \) and \( \psi_j\circ\theta_u\circ \theta \) agree on the generators. Clearly they agree on \( v_{\bar \imath} \) because by~\eqref{eq:psi_i} for an 
arbitrary \( i \in \NN \setminus \{ \bar \imath \} \) we have
\begin{equation} \label{eq:distinguished i}
(\psi_i\circ\theta_u\circ \theta)(v_{\bar \imath}) = \psi_i(v_{\rho(\bar \imath)}^{\bar{k}})=
v_{\rho(\bar \imath)}^{m}v_{\rho(\bar \imath)}^{\bar{k}} v_{\rho(\bar \imath)}^{-m} = v_{\rho(\bar \imath)}^{\bar{k}}, 
\end{equation}
independently of the integer \( m \) in the definition of \( \psi_i \).
 Next let \( i,j \in \NN \setminus \{ \bar{\imath} \} \) be arbitrary. 
By~\eqref{eq:auto} and~\eqref{eq:psi_i} one has
\begin{equation*} \label{eq:fund}
(\psi_i\circ\theta_u\circ \theta)(v_iv_j)=v_{\rho(i)}^{\bar{k}}v_{\rho(\bar{\imath})}^p v_{\rho(j)}^{\bar{k}} v_{\rho(\bar{\imath})}^{-p}
\end{equation*} 
for some $p \in \mathbb{Z}$ because 
\[ 
(\psi_i\circ\theta_u\circ \theta)(v_j) = (\psi_i \circ \psi_j^{-1} \circ \psi_j \circ\theta_u\circ \theta)(v_j) = (\psi_i \circ \psi_j^{-1})(v_{\rho(j)}^{\bar{k}}) .
 \] 
If $p\neq 0$ then \( v_{\rho(i)}^{\bar{k}}v_{\rho(\bar{\imath})}^p v_{\rho(j)}^{\bar{k}} v_{\rho(\bar{\imath})}^{-p} \) would have infinite order because \( \rho(i) \neq  \rho(\bar \imath) \) and \( \rho(j) \neq  \rho(\bar \imath) \): but since the order of \( v_i v_j \) is finite and \( \psi_i \circ \theta_u \circ \theta \in \Aut{G_T} \), this cannot be the case. Therefore \( p = 0 \) and 
\[
(\psi_i\circ\theta_u\circ \theta)(v_{j})= v_{\rho(j)}^{\bar{k}} =(\psi_j\circ\theta_u\circ \theta)(v_j).
\]  
Since all the \( \psi_i \) are the same,  by~\eqref{eq:auto} and~\eqref{eq:psi_i} there is a fixed  
\( m\in \mathbb{Z} \) (independent of \( i \)) such that \( (\theta_u \circ \theta)(v_i) = v_{\rho( \bar \imath)}^{m} v_{\rho(i)}^{\bar{k}} v_{\rho( \bar \imath)}^{-m} \), so that
\begin{equation} \label{eq:final}
\theta(v_i)= u  v_{\rho(\bar \imath)}^{m} v_{\rho(i)}^{\bar{k}} v_{\rho(\bar \imath)}^{-m} u^{-1}
 \end{equation}
 for every \( i \in \NN \setminus \{ \bar \imath \} \). But as observed at the beginning of this proof, equation~\eqref{eq:final} holds also for \( i = \bar \imath \), hence we are done.
\end{proof}

Consider now $\theta(v_0)$.
 Since $\theta(v_0)$ must have order $7$ in $G_T$,
Theorem~\ref{Theorem : Greendlinger} implies that there are some
\( n\in \mathbb{N} \) and
 \(w\in G_T \) such that
$\theta(v_0)=wv_{n}^{\bar{k}}w^{-1}$ with $\bar{k}\in\mathbb{Z}$
such that ${|\bar{k}|}<7$. Therefore we can apply Claim~\ref{claim:fundamental} with \( \bar{\imath} = 0 \),\( \bar \jmath =n \), and \( u = w \) to get a map \( \rho \) such that condition~\ref{c2} of the lemma is satisfied for \( \epsilon = \bar{k} \) and \( t=u v_{\rho(0)}^{m}\): thus it only remains to show that \( \rho \in \Aut{T} \). 
 
First observe that the orders of \( v_i v_j \) and \(v_{\rho(i)}v_{\rho(j)} \) are both equal to the order of \( v_{\rho(i)}^\epsilon v_{\rho(j)}^\epsilon \). In the former case one can use the fact that \( (\theta_t \circ \theta)(v_i v_j) = v_{\rho(i)}^\epsilon v_{\rho(j)}^\epsilon \) and that \( \theta_t \circ \theta \) is a group automorphism, where \( \theta_t \) is the map \( g \mapsto t^{-1} g t \). In the latter case, if \( \epsilon = -1 \) one can use the fact that \( v_{\rho(i)}v_{\rho(j)} \) and \( v_{\rho(j)} v_{\rho(i)} \) have the same order because the edge relation of a graph is symmetric, and that the latter has the same order of \( (v_{\rho(j)} v_{\rho(i)})^{-1} = v_{\rho(i)}^{-1} v_{\rho(j)}^{-1} \). Thus \( v_i v_j \) and \( v_{\rho(i)}v_{\rho(j)} \) have the same order. In particular, \( \rho \) is injective because if \( \rho(i) = \rho(j) \) then \( v_{\rho(i)} v_{\rho(j)} \) has order \( 7 \), so that \( v_i v_j \) has order \( 7 \) as well, and thus \( i = j \) by  definition of \( R_T \). 
Moreover 
\begin{align*}
i,j \text{ are adjacent in } T & \iff v_iv_j \text{ has order } 11 \\
& \iff v_{\rho(i)} v_{\rho(j)} \text{ has order } 11 \\
& \iff \rho(i), \rho(j) \text{ are adjacent in } T.
\end{align*} 

Finally, we show that \( \rho \) is surjective, i.e.\ that for every \( n \in \NN \) there is \( i \in \NN \) with \( n = \rho(i) \). First notice that if \(
v_{n} \) is conjugate to a power of \( v_{m} \), then \( n=m\). Indeed, if \( v_n = u v^k_m u^{-1}\) for some \( u \in G_T \) and \(k \in \mathbb{Z} \), then \( v_m v_n = v_m u v^k_m u^{-1}\). 
It follows that \( u \) is a power of \( v_m \), because otherwise  \(v_m u v^k_m u^{-1}\) would have infinite order, contradicting the fact that \( v_m v_n \) has finite order by definition of \( R_T \): therefore
\( v_m v_n =  v_m u v^k_m u^{-1} = v_m v^k_m \) (which also implies \( k \neq -1, 6 \) because \( v_m v_n \) is not the identity). But \( v_m v^k_m \) can only have order \( 7 \), whence \( n=m\).
Now fix an arbitrary \( n \in \NN \). Since  \( \theta^{-1}( v_{n}) \) has order \( 7 \), by Theorem~\ref{Theorem : Greendlinger} 
there are \( i\in \mathbb{N} \), \( u\in G_{T} \), and \( k' \in \mathbb{Z} \) such that
\( \theta^{-1}(v_{n})=u v_{i}^{k'}u^{-1} \), whence \( v_n = \theta(u) \theta(v_i)^{k'} \theta(u)^{-1} \). 
On the other hand, \( \theta(v_{i})= t v^{\epsilon}_{\rho(i)} t^{-1} \) by Claim \ref{claim:fundamental}, and substituting this value of \( \theta(v_i) \) in the previous equation one sees that \( v_n \) is conjugate to the \( (\epsilon k')\)-th power of \( v_{\rho(i)} \): thus \( n = \rho(i) \) by the observation above.

For the converse implication in Lemma~\ref{Lemma : auto}, assume that \( \theta \) satisfies \ref{c1} and \ref{c2}. Since \ref{c1} states that \( \theta \) is a group homomorphism, we are left with proving that  \( \theta \) is a bijection. Consider the inner automorphism \( \theta_{t} \), where \( t \) is as in~\ref{c2}, sending \( g\) to \( t^{-1} g t \), so that \( (\theta_t \circ \theta)( v_{i}) =  v_{\rho(i)}^{\epsilon} \) for every \( i \in \NN \): it clearly suffices to show that \( \theta_{t} \circ \theta \)
is a bijection.
For every nontrivial \( w = v_{i_{0}}\dots v_{i_{n}} \in G_{T} \) one has
\[
(\theta_{t} \circ \theta) (v^{\epsilon}_{\rho^{-1}(i_{0})}\dots v^{\epsilon}_{\rho^{-1}(i_{n})})=w,
\]
therefore \( \theta_{t} \circ \theta \) is surjective. 
As for injectivity, recall from the proof of~\cite[Theorem 5.1]{Wil14} that since \( \rho \) is an automorphism of \( T \), then the map \( \theta' \) induced by \( v_i \mapsto v_{\rho(i)} \) is an injection from \( G_T \) into itself. Thus if \( \epsilon = 1 \) we are done because \( \theta_t \circ \theta = \theta' \); if instead \( \epsilon = -1 \), then \( \theta_t \circ \theta \) is the composition of \( \theta' \) with the map induced by \( v_i \mapsto v_i^{-1} \), and since the latter is clearly injective we are done again.
\end{proof}

\begin{remark} \label{rmk:reductioniso}
Let \( T,S \in X_{Gr} \). If \( \rho \colon T \to S \) is an isomorphism, \( t \in G_S \), and \( \epsilon \in \{ -1 , 1 \} \), then the natural extension to the whole \( G_T \) of the map 
\begin{equation} \label{eq:iso}
\theta(v_i) = t v_{\rho(i)}^\epsilon t^{-1}.
 \end{equation}
 is an isomorphism between \( G_T \) and \( G_S \). Conversely, the proof of Lemma~\ref{Lemma : auto} can be straightforwardly adapted to show that every isomorphism \( \theta \colon G_T  \to G_S \) is canonically induced 
 by some isomorphism \( \rho \colon T \to S \) as above, i.e.\ that there are \( t \in G_S \) and \( \epsilon \in \{ -1,1 \} \) such that \( \theta \) satisfies~\eqref{eq:iso}. In particular, this shows that 
 \( T \cong S \iff G_T \cong G_S \).
\end{remark}

We are now ready to prove the main result of this section.

\begin{theorem}\label{Theorem : Gp uni}
The relation $\embeds_{Gp}$ is an invariantly universal  $\analytic$ quasi-order (when paired with the relation $\cong_{Gp}$ of isomorphism on countable groups). In particular,
for every \( \analytic \) quasi-order \( R \) there is an \( \L_{\omega_1 \omega} \)-elementary class of countable groups such that the embeddability relation on it is Borel bi-reducible with \( R \).
\end{theorem}

To fit the setup used in the above statement, each group $G_T$ must be coded as
an element \( \mathcal{G}_T\) of $X_{Gp}$ (the space of groups on \( \NN \)) via some bijection $\phi_T \colon G_T \xrightarrow[\text{onto}]{\text{1--1}} \mathbb{N}$.
In general, the specific coding is irrelevant, the only requirement being that the
map \(\mathcal{G}\) sending \(T\) to \(\mathcal{G}_{T}\), i.e.\ to the unique group isomorphic to \(G_{T}\) via \(\phi_{T}\),
be a Borel map from \(X_{Gr}\) to \(X_{Gp}\).
However, for our proof it is convenient to further require that for every $T\in X_{Gr}$, all generators  of $G_T$ and their inverses are sent by $\phi_T$ to some fixed natural numbers (independently of \( T \)),
and that for every reduced word $w$, all its subwords are sent by $\phi_T$ to numbers smaller than $\phi_T(w)$ (this technical conditions will be used in the proof of Proposition~\ref{Proposition : auto}). 
Thus for every $T\in X_{Gr}$ we fix a bijection $\phi_T:G_T\to \mathbb{N}$ such that
\begin{itemizenew}
\item
 $\phi_T(1_{G_T})=0$;
\item
 $\phi_T(v_i)=3i+1$;
\item 
$\phi_T(v_i^{-1})=3i+2$;
\item 
for every $n\in \mathbb{N}$ and for all subword $w$ of $\phi_T^{-1}(n)$, $\phi_T(w)<n$.
\end{itemizenew}
(Notice that words different from the identity, the generators and their inverses are sent to numbers of the form \( 3i+3 \).)

Let $\star_{T} \colon \NN\times\NN \to \NN$ be the binary operation on \( \NN \) 
such that
\( \mathcal{G}_T = (\NN,\star_T) \) is isomorphic to \( G_T \) via \( \phi_T \), that is: \( n\star_T m \coloneqq\phi_T(\phi_T^{-1}(n)\phi_T^{-1}(m)) \),
for every \( n, m \in \NN\).
Let \( \pre{< \NN}{(\NN) } \) be the set of all \emph{injective} \( t \in \pre{< \NN}{\NN} \), where \( \pre{< \NN}{\NN} \) is the set of finite sequences of natural numbers. 
Given $t\in \pre{<\NN}{(\NN)}$, let $N_t=\set{g\in S_\infty}{g\supseteq t}$. Clearly the set
\(
\set{N_t}{t \in \pre{< \NN}{(\NN)}}
 \) 
is a basis for \( S_\infty \).
Consider the maps
\[
\sigma \colon X_{Gr} \to \Subg {S_\infty}, \qquad T\mapsto \Aut T
\] 
and
\[
\Sigma \colon X_{Gr} \to \Subg{S_\infty}, \qquad T\mapsto \Aut {\mathcal{G}_T}.
\]

\begin{proposition}\label{Proposition : auto}
Let $T\in X_{Gr}$ and $s\in\pre{<\NN}{(\NN)}$. Then
\( \Sigma(T)\cap N_s\neq \emptyset\) if and only if the following conditions hold:
\begin{enumerate-(1)}
\item \label{condition1 : prop auto}
for every $n,m\in\dom(s)$, if \( n\star_Tm\in \dom (s) \) then \( s(n\star_T m)=s(n) \star_T s(m) \)
\item \label{condition2 : prop auto}
there is $r \colon \set{i}{3i+1\in \dom(s)}\to \NN$ such that
\begin{enumerate-(a)}
\item \label{item:a}
$\sigma(T)\cap N_r\neq\emptyset$
\item  \label{item:b}
there are $k,k'\in\mathbb{N}$ and  $l\in\{0,1\}$ such that \( k' \) is the inverse of \( k \) with respect to \( \star_T \) (i.e.\ \( k \star_T k'  = 0 \)) and 
\[\forall i\in\NN \, (3i+1\in\dom (s) \to s(3i+1)= k\star_T (3r(i)+1+l)\star_T k').\]
\end{enumerate-(a)}
\end{enumerate-(1)}
\end{proposition}

\begin{proof}
First assume that \( \Sigma(T)\cap N_s\neq \emptyset \), i.e.\ that 
there is some \( h \in\Aut{\mathcal{G}_T}\) such that \( h\supseteq s \).   
Since \(h \) is a homomorphism, if \( n,m \in \dom(s) \) are such that \( n\star_T m\in\dom (s) \) then
\[s(n\star_Tm)=h(n\star_Tm)=h(n)\star_Th(m)=s(n)\star_Ts(m),\]
which proves~\ref{condition1 : prop auto}. 
To prove \ref{condition2 : prop auto}, set $\theta\coloneqq\phi_T^{-1}\circ h\circ\phi_T$. Since $ \theta \in \Aut{G_T}$,
by Lemma~\ref{Lemma : auto} there are $\rho\in \Aut{T}$, $t\in G_T$, and $\epsilon\in\{-1,1\}$ such that for every $i\in\NN$
\[
\theta(v_i)=tv_{\rho(i)}^\epsilon t^{-1}.
\]
Setting $r= \rho\restriction \set{i\in\NN}{3i+1\in\dom (s)}$, one clearly
has $\rho\in \sigma(T)\cap N_r$, so that \(  \sigma(T)\cap N_r \neq \emptyset \). Moreover, setting $l\coloneqq - \frac{\epsilon -1 }{2}$, for every $i$ such that $3i+1\in s$
\[
s(3i+1)= (\phi_T\circ \theta)(v_i)=\phi_T(tv_{\rho(i)}^\epsilon t^{-1})=
k\star_T (3 r(i)+1+l)\star_{T} k',
\]
where $k=\phi_T(t)$ and \( k' = \phi_T(t^{-1}) \).

Conversely, assume that both~\ref{condition1 : prop auto} and~\ref{condition2 : prop auto} hold.
By~\ref{item:a} of condition~\ref{condition2 : prop auto} there is $\rho\in \Aut{T}$ such that $\rho\supseteq r$. Define 
\begin{align*}
h(0) & =0, \\
h(3i+1) & =k\star_T(3\rho(i)+1+l)\star_Tk', \\
h(3i+2) & =k\star_T(3\rho(i)+2-l)\star_Tk',
\end{align*}
and then extend \( h \) to the whole \( \NN \) via the operation \( \star_T \), i.e.\ if $n=\phi_T( v^{\epsilon_0}_{i_0}\dots v^{\epsilon_c}_{i_c})$ with \( \epsilon_0, \dotsc, \epsilon_c \in \{ -1,1 \} \), set
\( h(n)  =h(\phi_T(v^{\epsilon_0}_{i_0})) \star_T\dotsc \star_T h(\phi_T(v^{\epsilon_c}_{i_c}))$.
By~\ref{item:b} of condition~\ref{condition2 : prop auto}, the maps \(h\) and \(s\) agree on the codes for generators. Moreover,
 the way \(\phi_{T}\) was defined ensures that
if \(n=\phi_T( v^{\epsilon_0}_{i_0}\dots v^{\epsilon_c}_{i_c})\) belongs to
\(\dom(s)\), then so do all of \(\phi_{T}(v^{\epsilon_0}_{i_0}) \), \( \dotsc \), \(\phi_{T}(v^{\epsilon_c}_{i_c})\); thus \(h\supseteq s\) by
condition~\ref{condition1 : prop auto}.
Then one easily checks that $\theta\coloneqq \phi_T^{-1}\circ h\circ \phi_T$ satisfies \ref{c1}--\ref{c2} of Lemma \ref{Lemma : auto} with the chosen \( \rho \), \(t=\phi_T^{-1}(k)\), and \( \epsilon= 1-2l\).
Therefore \( \theta \in \Aut{G_T}\), whence $h$ is an automorphism of $\mathcal{G}_T$ witnessing \(  \Sigma(T)\cap N_s\neq \emptyset \).
\end{proof}

\begin{corollary}\label{Corollary : Sigma Borel}
Let \( B \subseteq X_{Gr} \) be a Borel set.
If $\sigma \restriction B$ is Borel, then $\Sigma \restriction B$ is Borel as well.
\end{corollary}

\begin{proof}
For \( s \in \pre{<\NN}{(\NN)} \), the preimage under \( \Sigma \restriction B \) of the generator \( \set{\mathbf{G} \in \Subg{S_\infty}}{\mathbf{G} \cap N_s \neq \emptyset} \) of the Effros Borel structure 
of \( \Subg{S_\infty} \) is \( \set{ T \in B }{ \Sigma(T) \cap N_s \neq \emptyset }\). By Proposition~\ref{Proposition : auto}, this is the set of graphs  \( T \in B \) satisfying 
conditions~\ref{condition1 : prop auto}--\ref{condition2 : prop auto} of Proposition~\ref{Proposition : auto}, which are all readily Borel with the possible exception of part~\ref{item:a} of condition~\ref{condition2 : prop auto}: 
but if \( \sigma \restriction B \) is a Borel map, then also that one becomes Borel, hence we are done.
\end{proof}

\begin{proof}[Proof of Theorem \ref{Theorem : Gp uni}]
It is enough to show that $(\embeds_{Gp},\cong_{Gp})$ satisfies conditions~\ref{condition : 1}--\ref{condition : 3} of Theorem \ref{Theorem : CMMR13}.
Since \(\mathcal{G}_{T}\) and \(G_{T}\) are isomorphic for every \(T\in X_{Gr}\), 
the map $f \colon \mathbb{G}\rightarrow X_{Gp},\, T\mapsto \mathcal{G}_T$
reduces \( \embeds_\mathbb{G} \) to \( \embeds_{Gp} \) by
Theorem~\ref{Theorem : Wil14} and thus \ref{condition : 1} is proved.
Part~\ref{condition : 2} follows from
the fact that $=_{\mathbb{G}}$ and $\cong_{\mathbb{G}}$ coincide and from
 Remark~\ref{rmk:reductioniso}, which still holds after replacing \(G_{T}\) with \(\mathcal{G}_{T}\).

Finally, we prove~\ref{condition : 3}. Since \( X_{Gp} \) is a space of countable structures and we are considering the isomorphism relation \( \cong_{Gp} \) on it, we are in the simplified situation described after
Theorem~\ref{Theorem : CMMR13}, so that it suffices to show that the map \( \Sigma \restriction \mathbb{G} \colon \mathbb{G} \to \Subg{S_\infty}, \, T\mapsto \Aut{\mathcal{G}_T} \) is Borel. Since every $T\in \mathbb{G}$ is rigid, the map $\sigma \restriction \mathbb{G} \colon \mathbb{G} \to \Subg{S_\infty}, \,  T\mapsto \Aut T$ is constant, hence Borel.
Therefore $\Sigma \restriction \mathbb{G}$ is Borel as well by Corollary~\ref{Corollary : Sigma Borel} and we are done.
\end{proof}

\section{Topological groups} \label{sec:topgroups}

In this section we study two different quasi-orders between topological groups.
The reduction defined by Williams in Theorem \ref{Theorem : Wil14} plays a key role, but it is convenient to encode the groups \( G_T \) in
a different standard Borel space. This variation allows us to prove the main theorems of this section in a simpler and direct way.

The \emph{countable random graph} \( R_\omega \) (see \cite{Rad64}) is a countable graph such that for any two finite sets \( A,B \) of vertices, there is a vertex \( x \) such that
\[
\forall y \in A \, (x \mathrel{R_\omega} y) \wedge \forall z \in B \, \neg (x \mathrel{R_\omega} z).
\]
An explicit definition of \( R_\omega \) (up to isomorphism) is the following: fix an enumeration of all prime numbers \( \set{p_n}{n\in\NN} \) and set for every 
 \( m,n\in\NN \setminus \{0,1 \}\)
\[
m\mathbin{R_\omega}n\quad\Leftrightarrow\quad {{p_m\mid n}\vee{p_n\mid m}}.
\]
Notice that each \( T\in X_{Gr} \) can be embedded into \( R_\omega \) in such a way that the map \( X_{Gr} \to {2}^{R_\omega}, T \mapsto  T' \) associating to every \( T \) an isomorphic subgraph \( T' \) of \( R_\omega \) is continuous. (This can be done due to the property which defines \( R_\omega\).)

Given $T\in X_{Gr}$, let $G_T$ be the group associated to $T$ defined as in the previous section (see the praragraph after Theorem~\ref{Theorem : Wil14}). Let $\SG{G_{R_\omega}}$  be set of all subgroups 
of $G_{R_\omega}$.\footnote{The space \(\SG{ G_{R_\omega} }\) is different from \( \Subg{\mathbf{G}}\). While \(\SG{G_{R_\omega} }\) is defined as the space of \emph{all} subgroups of \( G_{R_\omega} \), the space \(\Subg{\mathbf{G}}\) is the space of \emph{closed} subgroups of the Polish group \(\boldsymbol{G}\). }
The space \( \SG{G_{R_\omega}} \) can be construed as a closed subset of \( {2}^{G_{R_\omega}} \) by identifying each group with the characteristic function of its domain, and thus it is a Polish space with the induced topology inherited from \( {2}^{G_{R_\omega}} \).
Consider the variant of \( G \)

\[
\widetilde G \colon X_{Gr} \to \SG{G_{R_\omega}}, \qquad T \mapsto \widetilde G_T,
\] 
where \( \widetilde G_T \) is the subgroup of $G_{R_\omega}$ (isomorphic to $G_T$) whose generators are those appearing in $T \subseteq R_\omega$. Notice that the map \( \widetilde G \) is Borel as well.

Given a class $\mathbb{H}$ of Polish groups, we say that $\boldsymbol{W} \in\mathbb{H}$ 
is \emph{universal} (for \( \mathbb{H} \)) if every $\boldsymbol{H} \in \mathbb{H}$ topologically embeds into $\boldsymbol{W}$. 
The subsequent lemma will be used (twice) to define Borel reductions with target in hyperspaces of topological groups.  In the following, we turn \( G_{R_\omega} \) into a topological group \( \boldsymbol{G}_{R_\omega} \) by endowing it with the discrete topology, and every subgroup \( H \) of \( G_{R_\omega} \) in the corresponding (discrete) topological subgroup \( \boldsymbol{H} \) of \( \boldsymbol{G}_{R_\omega} \) (in particular, \( \widetilde{\boldsymbol{G}}_T \) is obtained by endowing \( \widetilde{G}_T \) with the discrete topology).

\begin{lemma}\label{Lemma : sec 4}
Let $\mathbb{H}$ be a standard Borel space of Polish groups, and assume that there is a universal group $\boldsymbol{W}\in  \mathbb{H}$.
If $\varphi \colon \boldsymbol{G}_{R_\omega}\to \boldsymbol{W}$ is a (topological) embedding into $\boldsymbol{W}$,
then the map
\[
 X_{Gr} \to \Subg{\boldsymbol{W}}, \qquad T \to \varphi[\widetilde {G}_T]
\]
is Borel.
\end{lemma}

\begin{proof}
Since \( \widetilde{G} \) is Borel, it is enough to prove that the function \( SG(G_{R_\omega}) \to \Subg{\boldsymbol{W}} \) mapping \( H \) to \( \varphi[\boldsymbol{H}] \) is Borel, i.e.\ that given a nonempty open set $U\subseteq \boldsymbol{W}$, the preimage of $B_U=\set{\boldsymbol{F}\in \Subg{\boldsymbol{W}}}{\boldsymbol{F}\cap U\neq \emptyset}$ is a Borel subset of $\SG{G_{R_\omega}}$. 
This is clear, as for every \( H \in SG(G_{R_\omega}) \) one has \( \varphi[\boldsymbol{H}] \in B_U \) if and only if \( h \in C_U \) for some \( h \in H \), where \( C_U = \set{g \in G_{R_\omega}}{\varphi(g) \in U} \).
\end{proof}

\subsection{Polish groups} \label{sec:polishgroups}

We denote by $X_{PGp}$ the hyperspace of all Polish groups, which may be construed as follows. It is well known that there are Polish groups \( \boldsymbol{W} \) which are universal, i.e.\ such that all Polish groups topologically embed into \( \boldsymbol{W} \). For example, one may let \( \boldsymbol{W} \) be the Polish group $\Homeo{[0,1]^\NN}$ of all homeomorphisms of the Hilbert cube
 (see e.g.~\cite[Theorem 9.18]{Kec}), or the Polish group $\Isom{\mathbb{U}}$ of isometries of the Urysohn space $\mathbb{U}$ (see \cite[Theorem 2.5.2]{Gao}). For the sake
of definiteness, we set \( \boldsymbol{W} = \Homeo{[0,1]^\NN} \) so that we may let \( X_{PGp} \) be the standard Borel space
\[
\Subg{\Homeo{[0,1]^\NN}}.
\]

Given two Polish groups $\boldsymbol H$ and $\boldsymbol H'$, we  write $\boldsymbol H \embeds_{PGp} \boldsymbol H'$ when  $\boldsymbol H$ topologically embeds into $\boldsymbol H'$. In the next theorem we give an alternative proof of the fact that topological embeddability between Polish groups is complete (compare this with~\cite[Corollary 34]{FerLouRos}).

\begin{theorem}\label{Theorem : PGp-complete}
The relation $\embeds_{PGp}$ is a complete $\analytic$ quasi-order.
\end{theorem}

\begin{proof}
By Theorem~\ref{Theorem : LouRos}, it suffices to show that \( {\embeds_{Gr}} \leq_B {\embeds_{PGp}} \).
Since \( \Homeo{[0,1]^\NN} \) is universal,  there is a topological embedding $\varphi \colon \boldsymbol G_{R_\omega}\to\Homeo{[0,1]^\NN}$.
 Consider the map
\begin{equation} \label{eq:f}
f \colon X_{Gr} \to X_{PGp}, \qquad T \mapsto \varphi[\boldsymbol{\widetilde G}_T],
\end{equation}
which is Borel by Lemma~\ref{Lemma : sec 4}.
Since every function between discrete Polish groups is continuous and $\widetilde G_T$
is isomorphic to $G_T$, one has that for every $T,S\in X_{Gr}$
\[
T \embeds_{Gr} S \iff \widetilde G_T \embeds_{Gp} \widetilde G_S \iff \boldsymbol{\widetilde G}_T \embeds_{PGp} \boldsymbol{\widetilde G}_S \iff  f(T) \embeds_{PGp} f(S),
\]
hence $f$ reduces $\embeds_{Gr}$ to $\embeds_{PGp}$.
\end{proof}

\begin{remark}
Notice that our proof of Theorem~\ref{Theorem : PGp-complete} uses non-Abelian groups, while~\cite[Corollary 34]{FerLouRos} further shows that the topological embeddability
between Abelian Polish groups is complete as well.
\end{remark}

\begin{theorem}\label{Theorem : PGp uni}
The relation $\embeds_{PGp}$ is invariantly universal (when paired with the relation of topological isomorphim \( \cong_{PGp} \)).
\end{theorem}

\begin{proof}
It is enough to show that the pair  $(\embeds_{Gp},\cong_{Gp})$ satisfies conditions~\ref{condition : 1}--\ref{condition : 3} of Theorem \ref{Theorem : CMMR13}.
Set $g = f \restriction \mathbb{G}$, where \( f \) is as in~\eqref{eq:f}. 
We already proved that $g$ reduces $\embeds_{\mathbb{G}}$ to $\embeds_{PGp}$ in Theorem \ref{Theorem : PGp-complete}, hence \ref{condition : 1} holds.
To see \ref{condition : 2}, notice that $g$ witnesses that ${=_{\mathbb{G}}}\leq_{B} {\cong_{PGp}}$. In fact, since ${=_\mathbb{G}}\leq_B {\cong_{Gp}}$ (cf. Theorem \ref{Theorem : Gp uni}) and each \( G_T \) is isomorphic to \(\widetilde G_T\), then for every \( T,S\in X_{Gr}\), 
\[
T =_{\mathbb{G}} S \iff \widetilde G_T \cong_{Gp} \widetilde G_S \iff \boldsymbol{\widetilde G}_T \cong_{PGp} \boldsymbol{\widetilde G}_S \iff  g(T) \cong_{PGp} g(S),
\]
where the second equivalence holds because
every function between discrete Polish groups is continuous.

Finally, we prove \ref{condition : 3}. Let  $(\psi_k)_{k\in \NN}$ be a sequence  of Borel selectors for \( X_{PGp} \), i.e.\ each \( \psi_k \) is a function from \( X_{PGp} = \Subg{\Homeo{[0,1]^\NN}} \) to \( \Homeo{[0,1]^\NN} \) such that \( \psi_k(\boldsymbol{G}) \in \boldsymbol{G} \) for every \( \boldsymbol{G} \in \Subg{\Homeo{[0,1]^\NN}}  \), and for every such \( \boldsymbol{G} \) the set \( \set{\psi_k(\boldsymbol{G})}{k \in \NN} \) is dense in \( \boldsymbol{G} \). Recall that we may assume that \( \psi_k(\boldsymbol{G}) \neq \psi_{k'}(\boldsymbol{G}) \) for all \( k \neq k' \) whenever \( \boldsymbol{G} \) is infinite.
Let $\set{U_n}{n\in\mathbb{N}}$ be a countable basis for the topology of $\Homeo{[0,1]^\mathbb{N}}$. Let
\begin{multline} \label{eq:Z}
Z=\set{\boldsymbol{G}\in X_{PGp}}{\forall k \forall k' \, (k \neq k' \to \psi_k(\boldsymbol{G}) \neq \psi_{k'}(\boldsymbol{G})) \\
\wedge \, \exists n \, ( 1_{\boldsymbol{G}} \in U_n \wedge \forall k\, ( \psi_k(\boldsymbol{G}) \in U_n \rightarrow \psi_k(\boldsymbol{G}) = 1_{\boldsymbol{G}}))
},
\end{multline}
where \( 1_{\boldsymbol{G}} \) is the identity of \( \boldsymbol{G} \). Notice that  every \( \boldsymbol G \in Z \) needs to be infinite because all its elements of the form \( \psi_k(\boldsymbol{G}) \) are distinct.
We also claim that if $\boldsymbol{G}\in Z$, then $1_{\boldsymbol{G}}$ is an isolated point.
In fact, if this is not the case then for every \( n \in \NN \) such that \( 1_{\boldsymbol{G}} \in U_n \) there would be $x\neq 1_{\boldsymbol{G}}$ such that $x \in U_n$. Since \( \boldsymbol{G} \) is Hausdorff, one could then pick some open set $V$ with $x\in V$ and $1_{\boldsymbol{G}}\notin V$. Since the point \( x \) witnesses that the open set $V \cap U_n$ is nonempty, there would be some
$\psi_k(\boldsymbol{G}) \in V \cap U_n$, which is necessarily distinct from \(1_{\boldsymbol{G}}$ because $1_{\boldsymbol{G}} \notin V$. But then  
$\psi_k(\boldsymbol{G}) \in V \cap U_n \subseteq U_n$ and $ \psi_k(\boldsymbol{G}) \neq 1_{\boldsymbol{G}}$. Since \( n \) was arbitrary, this contradicts \( \boldsymbol{G} \in Z \).
Since a topological group is discrete if and only if its unity is an isolated point, $G\in Z$ if and only if it is infinite and discrete. Therefore \( Z \)  is \(\cong_{PGp}\)-invariant and the definition given in~\eqref{eq:Z} directly shows that $Z$ is a Borel set.
 
 Let \( h \) be the forgetful map \( Z \to X_{Gp} \) associating to each \( \boldsymbol{G} \in Z \)
 the group \( h(\boldsymbol{G})=(\NN,\star_{\boldsymbol{G}})\) with underlying set \( \NN \) and \( \star_{\boldsymbol{G}} \) defined by setting
 \[
 k\star_{\boldsymbol G}m=n \iff \psi_k(\boldsymbol G )\psi_m(\boldsymbol G ) = \psi_n(\boldsymbol G ).
 \]
Now modify \( h \) by imposing that $h(g(T))=\mathcal{G}_T$ for every $T\in \mathbb{G}$, i.e.\ set
\( h(\boldsymbol H)\coloneqq \mathcal G_{g^{-1}(\boldsymbol H)}\) for every 
\( \boldsymbol H \in \mathrm{rng}(g)\).
Notice that the resulting map, which will be denoted again by \( h \), is still Borel because $g$ is a Borel injective map, whence $\mathrm{rng}(g)$ is a Borel subset of $Z$ and the map \( \mathrm{rng}(g) \to X_{Gp}, \, \boldsymbol{H} \mapsto \mathcal G_{g^{-1}(\boldsymbol H)} \), being the composition of the Borel maps \( g^{-1} \) and \( \mathcal{G} \), is Borel.
 Now consider the logic action of $S_\infty$ on $X_{Gp}$: the stabilizer of $h(g(T))$
with respect to this action is just $\Aut{h(g(T)}$, which equals $\Aut{\mathcal{G}_T}$ by the way we modified \( h \). Therefore the map
$T\mapsto \Aut{h(g(T)}$ is Borel by (the proof of) Theorem~\ref{Theorem : Gp uni} and we are done. 
\end{proof}

\subsection{Separable groups with bounded (bi-invariant) metric} \label{sec:metricgroups}

In this section we study the quasi-order of isometric embeddability between separable complete metric groups (briefly: \emph{Polish metric groups}) with bounded bi-invariant metric. In order to define the 
standard Borel hyperspace of (codings for) such groups we can use the existence
of a sufficiently universal
 object%
\footnote{Actually, the only property that we need is that \({\boldsymbol G}_{R_\omega}\) embeds into it --- see the proof of Theorem~\ref{Theorem : embeds_i complete}.} for this class. Recently Doucha proved the following theorem.

\begin{theorem}[{\cite[Theorem~1.1]{Dou}}]\label{Theorem : Doucha}
For every positive real $K>0$, there is a Polish metric group $\boldsymbol{D}_K$ with bi-invariant metric \(d_{K}\) bounded by $K$, which contains a closed isometric copy of every separable group with a complete bi-invariant metric bounded by $K$.
\end{theorem}

Therefore we can use \(\boldsymbol{D}_K\) as the universal object, and regard
\[ 
X^K_{PMGp} = \Subg{\boldsymbol{D}_K}
 \] 
as the standard Borel space of all Polish metric groups whose metric is bi-invariant and bounded by \( K \).

We say that $\boldsymbol H$ \emph{isometrically embeds} into $\boldsymbol H'$, and write $\boldsymbol H\embeds^K_i\boldsymbol H'$, if  there is an isometric group embedding from
$\boldsymbol H$ into $\boldsymbol H'$.

\begin{theorem}\label{Theorem : embeds_i complete}
For every $K>0$, the relation $\embeds_i^K$ is a complete \( \analytic \) quasi-order.
\end{theorem}

\begin{proof}
Fix $K>0$. Endow $\boldsymbol G_{R_\omega}$ with the discrete metric with value $K$, 
that is set \( d(x,y) = K \) for all distinct \( x,y \in \boldsymbol{G}_{R_\omega} \).
By Theorem~\ref{Theorem : Doucha} there exists an isometric embedding $\varphi \colon \boldsymbol G_{R_\omega}\to \boldsymbol D_K$.  Let $f \colon X_{Gr}\to\Subg{\boldsymbol D_K}$ be the map sending $T$
to  $\varphi[\widetilde {\boldsymbol G}_T]$: we claim that \( f \) Borel reduces \( \embeds_{Gr} \) to \( \embeds^K_i \), so that the result follows from Theorem~\ref{Theorem : LouRos}.

By Lemma~\ref{Lemma : sec 4} the map $f$ is Borel.
Notice that each $f(T)$ is isomorphic to $G_T$ when viewed as a countable structure, i.e.\ when forgetting the metric and the resulting topology. Since any one-to-one function between groups in the range of $f$ is automatically an isometry (because all such groups are equipped with the discrete metric with constant value $K$), we have that for every \( T,S \in X_{Gr} \)
\[ 
T \embeds_{Gr} S \iff G_T \embeds_{Gp} G_S \iff f(T) \embeds^K_i f(S). \qedhere
 \] 
\end{proof}

\begin{theorem}\label{Theorem : embeds_i invariantly universal}
For every $K>0$, the relation  $\embeds_i^K$ is invariantly universal (when paired with the isometric isomorphism \( \cong_i^K \) on \( X^K_{PMGp} \)).
\end{theorem}
\begin{proof}
Fix $K>0$. Let $g$ be the restriction to \( \mathbb{G} \) of the map \( f \) defined in the proof of \ref{Theorem : embeds_i complete}. It suffices to show that conditions~\ref{condition : 1}--\ref{condition : 3} 
of Theorem~\ref{Theorem : CMMR13}
are satisfied. 
The fact that $g$ reduces $\embeds_{\mathbb{G}}$ to $\embeds_{PGp}$ is proved in Theorem~\ref{Theorem : PGp-complete}, hence condition~\ref{condition : 1} is fulfilled.
Notice that $g$ also witnesses that $=_{G}$ Borel reduces to $\cong_K^i$ (condition~\ref{condition : 2}). Indeed, for every \( T,S \in X_{Gr} \) 
\[
T=_{\mathbb{G}} S \iff  G_T\cong_{Gp}G_{S} \iff g(T)\cong_K^i g(S),
\]
where the former equivalence follows from the proof of Theorem \ref{Theorem : Gp uni}, while  
the latter equivalence holds because \( g(T) \) is isomorphic to \( G_T \) as a group, and the metric
of $g(T)$ is discrete with the same constant value for every $T\in \mathbb{G}$. 

Finally, we prove that also condition~\ref{condition : 3} holds.
Let $(\psi_i)_{i\in\NN}$ be a sequence of Borel selectors for the Polish subgroups of $\boldsymbol D_K$, so that  for every nonempty \( \boldsymbol{H} \in \Subg{\boldsymbol{D}_K} = X^K_{PMGp} \) the sequence \( ( \psi_i(\boldsymbol{H}) )_{ i \in \NN } \) is an enumeration (without repetitions if \( \boldsymbol{H} \) is infinite) of a dense subset of \( \boldsymbol{H}\).
Set
\[
Z=\set{\boldsymbol H\in X^K_{PMGp}} {d_K(\psi_n(\boldsymbol H),\psi_{m}(\boldsymbol H))=K \text{ for distinct }n,m \in \NN},
\]
where \(d_{K}\) is the metric of \(\boldsymbol{D}_K\).
It is immediate to check that $Z$ is a Borel subset of $ X^K_{PMGp}$.  Notice also that every $\boldsymbol H\in  X^K_{PMGp}$ is infinite, and actually it coincides with \( \set{\psi_i(\boldsymbol H)}{i \in \NN} \) because every point \( \psi_i(\boldsymbol H) \) is isolated in it (by the definition of \( Z \)). It follows that \( Z \) is also \( \cong_i^K \)-invariant.

Arguing as in the last paragraph of the proof of Theorem~\ref{Theorem : PGp uni}, modify the forgetful map $h \colon Z\to X_{Gp}$ so that \( h(g(T))=\mathcal G_T \) for every \( T \in \mathbb{G} \). The resulting map is Borel and 
 reduces \( \cong_{i}^K \) to \( \cong_{Gp} \). Moreover, the stabilizer  of each \( h(g(T)) \) with respect to the logic action is exactly \( \Aut{\mathcal{G}_T} \) by the definition of \( h \), therefore the map \( T \mapsto \Aut{h(g(T))} \) is Borel by the proof of Theorem~\ref{Theorem : Gp uni} and we are done.
\end{proof}

An alternative approach to study the isometric embeddability between Polish metric groups with a bounded bi-invariant metric is to use the setup of continuous logic (see~\cite{BenBerHenUsv}). In this context, each separable metric group $\boldsymbol{G} = (G,d_G)$ would be identified with a code \( \boldsymbol{c}_G = (c^0_G,c^1_G) \in \RR^{\NN^3} \times \RR^{\NN^2} \) by fixing a dense subgroup \( \set{g_i}{i \in \NN} \) of \( \boldsymbol{G} \) and setting for every \( i,j,k\in \NN \) 
\[ 
c^0_G(i,j,k) = d_G(g_i g_j, g_k) \qquad \text{and} \qquad c^1_G(i,j) = d_G(g_i,g_j).
\] 
The set \( \mathscr{G} \) of codes for Polish metric groups turns out to be a \( G_\delta \) subset of \( \mathcal{M}(\L) = \RR^{\NN^3} \times \RR^{\NN^2} \), the space of \( \L \)-structures (in continuous logic) of the 
language \( \L\) 
consisting of a ternary relation symbol (the one corresponding to the graph of the group operation) and a binary relation symbol (the one corresponding to the distance of the group) --- see~ \cite{BenDouNieTsa} for more on this. 
Theorem~\ref{Theorem : embeds_i invariantly universal} can be recasted in this setup as follows: for every \( K > 0 \) and every analytic quasi-order \( R \) there is a Borel set \( B \subseteq \mathscr{G} \) invariant under 
isomorphism (consisting of Polish metric groups with a bi-invariant metric bounded by \( K \)) such that \( R \) is Borel bi-reducible with the embeddability relation on \( B \). By the Lopez-Escobar theorem for continuous logic proved in~\cite{BenDouNieTsa}, we then get the following elegant reformulation of Theorem~\ref{Theorem : embeds_i invariantly universal} (compare it with the second part of Theorem~\ref{Theorem : Gp uni}). 

\begin{theorem} \label{thm:continuouslogic}
Let \( K > 0 \). Then for every analytic quasi-order \( R \) there is an  $\mathcal{L}_{\omega_1\omega}$-sentence \( \upvarphi \) of continuous logic all of whose models are Polish metric groups with bi-invariant metric bounded by \( K \) and such that \( R \) is Borel bi-reducible with the embeddability relation on the models of \( \upvarphi \).
\end{theorem}


\begin{thebibliography}{BYDNT16}

\bibitem[BYBHU08]{BenBerHenUsv}
Ita\"\i{} Ben~Yaacov, Alexander Berenstein, C.~Ward Henson, and Alexander
  Usvyatsov.
\newblock Model theory for metric structures.
\newblock In {\em Model theory with applications to algebra and analysis.
  {V}ol. 2}, volume 350 of {\em London Math. Soc. Lecture Note Ser.}, pages
  315--427. Cambridge Univ. Press, Cambridge, 2008.

\bibitem[BYDNT16]{BenDouNieTsa}
Ita\"\i{} Ben~Yaacov, Michal Doucha, Andre Nies, and Todor Tsankov.
\newblock Metric scott analysis.
\newblock preprint, \url{http://arxiv.org/abs/1407.7102}, 2016.

\bibitem[CMMR]{CamMarMot16}
Riccardo Camerlo, Alberto Marcone, and Luca Motto~Ros.
\newblock On isometry and isometric embeddability between metric and
  ultrametric polish spaces.
\newblock preprint, \url{https://arxiv.org/abs/1412.6659}.

\bibitem[CMMR13]{CamMarMot}
Riccardo Camerlo, Alberto Marcone, and Luca Motto~Ros.
\newblock Invariantly universal analytic quasi-orders.
\newblock {\em Trans. Amer. Math. Soc.}, 365(4):1901--1931, 2013.

\bibitem[Dou16]{Dou}
Michal Doucha.
\newblock Metrical universality for groups.
\newblock {\em Forum Math.}, 2016.
\newblock To appear, \url{https://doi.org/10.1515/forum-2015-0181}.

\bibitem[FLR09]{FerLouRos}
Valentin Ferenczi, Alain Louveau, and Christian Rosendal.
\newblock The complexity of classifying separable {B}anach spaces up to
  isomorphism.
\newblock {\em J. Lond. Math. Soc. (2)}, 79(2):323--345, 2009.

\bibitem[FMR11]{FriMot}
Sy-David Friedman and Luca Motto~Ros.
\newblock Analytic equivalence relations and bi-embeddability.
\newblock {\em J. Symbolic Logic}, 76(1):243--266, 2011.

\bibitem[FS89]{FriSta}
Harvey Friedman and Lee Stanley.
\newblock A {B}orel reducibility theory for classes of countable structures.
\newblock {\em J. Symbolic Logic}, 54(3):894--914, 1989.

\bibitem[Gao09]{Gao}
Su~Gao.
\newblock {\em Invariant descriptive set theory}, volume 293 of {\em Pure and
  Applied Mathematics (Boca Raton)}.
\newblock CRC Press, Boca Raton, FL, 2009.

\bibitem[GK03]{Gao2003}
Su~Gao and Alexander~S. Kechris.
\newblock On the classification of {P}olish metric spaces up to isometry.
\newblock {\em Mem. Amer. Math. Soc.}, 161(766):viii+78, 2003.

\bibitem[HKL90]{HarKecLou}
L.~A. Harrington, A.~S. Kechris, and A.~Louveau.
\newblock A {G}limm-{E}ffros dichotomy for {B}orel equivalence relations.
\newblock {\em J. Amer. Math. Soc.}, 3(4):903--928, 1990.

\bibitem[Kec95]{Kec}
Alexander~S. Kechris.
\newblock {\em Classical descriptive set theory}, volume 156 of {\em Graduate
  Texts in Mathematics}.
\newblock Springer-Verlag, New York, 1995.

\bibitem[LR05]{LouRos}
Alain Louveau and Christian Rosendal.
\newblock Complete analytic equivalence relations.
\newblock {\em Trans. Amer. Math. Soc.}, 357(12):4839--4866 (electronic), 2005.

\bibitem[LS01]{LynSch}
Roger~C. Lyndon and Paul~E. Schupp.
\newblock {\em Combinatorial group theory}.
\newblock Classics in Mathematics. Springer-Verlag, Berlin, 2001.
\newblock Reprint of the 1977 edition.

\bibitem[Rad64]{Rad64}
R.~Rado.
\newblock Universal graphs and universal functions.
\newblock {\em Acta Arith.}, 9:331--340, 1964.

\bibitem[Sab16]{Sab}
Marcin Sabok.
\newblock Completeness of the isomorphism problem for separable {$\rm
  C^\ast$}-algebras.
\newblock {\em Invent. Math.}, 204(3):833--868, 2016.

\bibitem[Wil14]{Wil14}
Jay Williams.
\newblock Universal countable {B}orel quasi-orders.
\newblock {\em J. Symb. Log.}, 79(3):928--954, 2014.

\end{thebibliography}
\end{document}